\documentclass[14pt]{amsart}

\usepackage{cases}
\usepackage{amsmath}
\usepackage{amsfonts}
\usepackage{bm}
\usepackage{amsfonts,amsmath,amssymb,amscd,bbm,amsthm,mathrsfs,dsfont}
\usepackage{mathrsfs}
\usepackage{pb-diagram}
\usepackage{amssymb}

\newtheorem{Theorem}{Theorem}[section]
\newtheorem{Lemma}[Theorem]{Lemma}

\newtheorem{Definition}[Theorem]{Definition}
\newtheorem{Corollary}[Theorem]{Corollary}
\newtheorem{Proposition}[Theorem]{Proposition}

\newtheorem{Remark}[Theorem]{Remark}
\newtheorem{Conjecture}[Theorem]{Conjecture}

\date{version of \today}

 \normalbaselines
\title
[on cluster algebras via abstract pattern and two conjectures]
{Study on cluster algebras via abstract pattern and \\two conjectures on {\bf d}-vectors and {\bf g}-vectors}

\author{Peigen Cao $\;\;\;\;\;\;$ Fang Li $\;\;\;\;\;\;$}
\address{Peigen Cao
\newline Department
of Mathematics, Zhejiang University (Yuquan Campus), Hangzhou, Zhejiang
310027,  P.R.China}
\email{peigencao@126.com}

\address{Fang Li
\newline Department
of Mathematics, Zhejiang University (Yuquan Campus), Hangzhou, Zhejiang
310027, P.R.China}
\email{fangli@zju.edu.cn}

\begin{document}
\renewcommand{\thefootnote}{\alph{footnote}}

\renewcommand{\thefootnote}{\alph{footnote}}
\setcounter{footnote}{-1} 
\renewcommand{\thefootnote}{\alph{footnote}}
\setcounter{footnote}{-1} 

\begin{abstract}  We mainly introduce an abstract pattern to study cluster algebras.   Cluster algebras, generalized cluster algebras and Laurent phenomenon algebras are unified in the language of generalized Laurent phenomenon algebras (briefly, GLP algebras) from the perspective of Laurent phenomenon.

 In this general framework, we firstly prove that each positive and  {\bf d}-vector-positive GLP algebra has the proper  Laurent monomial  property and thus its cluster monomials are linearly independent.   Skew-symmetric cluster algebras are verified to be {\bf d}-vector-positive, which gives the affirmation of Conjecture \ref{conjd} in \cite{FZ3} in this case. And since the positivity of skew-symmetric cluster algebras is well-known,  the new proof is obtained for the linearly independence of cluster monomials of  skew-symmetric cluster algebras.

For a class of GLP algebras which are  {\em pointed}, {\bf g}-vectors of cluster monomials are defined. We verify  that for any positive GLP algebra pointed at $t_0$, the {\bf g}-vectors ${\bf g}_{1;t}^{t_0},\cdots, {\bf g}_{n;t}^{t_0}$  form a $\mathbb Z$-basis of $\mathbb Z^n$, and  different cluster monomials have different {\bf g}-vectors. As a direct application, we verify that Conjecture \ref{conjg} in \cite{FZ3} holds for skew-symmetrizable cluster algebras and  acyclic sign-skew-symmetric cluster algebras.

\end{abstract}

\maketitle
\bigskip

\section{introduction}

Cluster algebras were introduced by Fomin and Zelevinsky in \cite{FZ}. The motivation was to create a common framework for phenomena occurring in connection
with total positivity and canonical bases. A cluster algebra $\mathcal A$ of rank $n$ is a subalgebra of an ambient field $\mathcal F$ generated by certain combinatorially defined generators (called as {\em cluster variables}). One of the important features of cluster algebras is that they have Laurent phenomenon. Thus one can define the {\em {\bf d}-vectors} of cluster variables. Fomin and Zelevinsky conjectured that
\begin{Conjecture}\label{conjd}
(Conjecture 7.4 (i) in \cite{FZ3}: Positivity of {\bf d}-vectors) For any given a cluster $X_{t_0}$ of a cluster algebra $\mathcal A$, and any cluster variable  $x\notin X_{t_0}$,  the denominator vector (briefly, {\bf d}-vector) $d^{t_0}(x)$ of  $x$ with respect to $X_{t_0}$ is in $\mathbb N^n$.
\end{Conjecture}

We say a cluster algebra $\mathcal A$ to be {\bf d-vector-positive} if Conjecture \ref{conjd} holds for  $\mathcal A$ (see Definition \ref{defd}).
It is known that cluster algebras are  {\bf d}-vector-positive in the following cases:
 cluster algebras of rank $2$ (Theorem 6.1 of \cite{FZ}),  cluster algebras arising from surfaces (see \cite{FST}), and  cluster algebras of finite type (see \cite{CCS,CP}).

  Caldero and  Keller proved a weak version of Conjecture \ref{conjd} in \cite{CK}, that is, for a skew-symmetric cluster algebra $\mathcal A$, if the seed $\Sigma(X_{t_0})$ at $t_0$ is acyclic, then the  {\bf d}-vector $d^{t_0}(x)$ of a cluster variable $x\notin X_{t_0}$ is in $\mathbb N^n$.

 In this paper, we will verify Conjecture \ref{conjd} for any skew-symmetric cluster algebras,
that is,  skew-symmetric cluster algebras are always {\bf d}-vector-positive (Theorem \ref{thmdgood}).

  The cluster monomial in a cluster algebra $\mathcal A$ is a monomial in cluster variables from the same cluster.  Cluster monomials are conjectured to be linearly independent over the coefficient ring $\mathbb {ZP}$. In \cite{CLF}, the authors prove that if a cluster algebra has {\em the  proper  Laurent monomial property} (see Definition \ref{deflaurent}), then its cluster monomials are  linearly independent. In \cite{CKLP}, the authors prove the skew-symmetric algebras are always having {\em the proper  Laurent monomial property}, by using the method of representation theory and thus they obtain the linear independence of cluster monomials for skew-symmetric cluster algebras.

The definitions of both generalized cluster algebras  and Laurent phenomenon algebras are the generalizations of  that  of cluster algebras from the perspective of exchange relations. They have the so-called {\em Laurent phenomenon} (see \cite{CS} and \cite{LP}). From \cite{CS,NT,CL}, we know that generalized cluster algebras and cluster algebras share many common properties. Lam and Pylyavskyy in \cite{LP} believed that some features of cluster algebras can be extended to Laurent phenomenon algebras.
 Inspired by these, we consider the fundamental characterization of cluster algebras as cluster patterns on $n$-regular tree with the Laurent phenomenon to introduce an abstract pattern, that is, the generalized Laurent phenomenon algebras (or say, GLP algebras) as a general framework in order to concentrate on those common  properties of such algebras having the Laurent phenomenon.

In this paper, we prove that if the positivity of  {\bf d}-vectors and the positivity of cluster variables hold for a GLP algebra $\mathcal A(M_T)$ (including skew-symmetric cluster algebras as the special case), then $\mathcal A(M_T)$ has the  proper  Laurent monomial property  and thus the cluster monomials of $\mathcal A(M_T)$ are linearly independent over the coefficient ring $\mathbb{ZP}$ (Theorem \ref{mainthm}).

  Cluster algebras with principal coefficients are very important research objects in the theory of cluster algebras.  The following statements are conjectured by Fomin and Zelevinsky in \cite{FZ3} for this kind of cluster algebras.

\begin{Conjecture}(Conjectre 7.10, \cite{FZ3})\label{conjg}
Let $\mathcal A$ be a  cluster algebras with principal coefficients. Then,

(i) Different cluster monomials have different {\em {\bf g}-vectors};

(ii) The {\em {\bf g}-vectors} ${\bf g}_{1;t}^{t_0},\cdots, {\bf g}_{n;t}^{t_0}$ form a $\mathbb Z$-basis of $\mathbb Z^n$.

\end{Conjecture}
This conjecture was verified for skew-symmetrizable cluster algebras, see  \cite{DWZ,NZ,GHKK} for details.

 Motivated by Theorem \ref{thmghkk}, we realize that it is reasonable to define {\em {\bf g}-vectors} and {\em G-matrix} in a general framework for GLP algebras pointed at certain cluster $X_{t_0}$ (Definition \ref{defprincipal}). And the  {\bf g}-vectors defined for pointed GLP algebras are  natural generalization of
{\bf g}-vectors for cluster algebras with principal coefficients.

In this paper, we prove that Conjecture \ref{conjg} holds for positive and pointed GLP algebras (see Corollary \ref{corconj}). As a direct application, we obtain the fact that  Conjecture \ref{conjg} holds for  skew-symmetrizable cluster algebras and acyclic sign-skew-symmetric cluster algebras.

This paper is organized as follows. In Section 2, some basic definitions and notations are introduced. In Section 3, we prove that if a  GLP algebra $\mathcal A(M_T)$  is both positive and {\bf d}-vector-positive, then it has the  proper Laurent monomial property (Theorem \ref{mainthm}) and thus its cluster monomials are linearly independent. In Section 4, we verify that skew-symmetric cluster algebras are {\bf d}-vector-positive (Theorem \ref{thmdgood}). In Section 5, we prove that Conjecture \ref{conjg} holds for positive and pointed GLP algebras (see Theorem \ref{thmmonomial}, Theorem \ref{detgthm} and Corollary \ref{corconj}).

\section{Preliminaries}
\subsection{cluster pattern and cluster algebra}

Recall that $(\mathbb P, \oplus, \cdot)$ is a {\bf semifield } if $(\mathbb P,  \cdot)$ is an abelian multiplicative group endowed with a binary operation of auxiliary addition $\oplus$ which is commutative, associative, and distributive with respect to the multiplication $\cdot$ in $\mathbb P$.

Let $Trop(u_i:i\in I)$ be a free abelian group generated by $\{u_i: i\in I\}$ for a finite set of index $I$. We define the addition $\oplus$ in $Trop(u_i:i\in I)$ by $\prod\limits_{i}u_i^{a_i}\oplus\prod\limits_{i}u_i^{b_i}=\prod\limits_{i}u_i^{min(a_i,b_i)}$, then $(Trop(u_i:i\in I), \oplus)$ is a semifield, which is called a {\bf tropical semifield}.

The multiplicative group of any semifield $\mathbb P$ is torsion-free for multiplication \cite{FZ}, hence its group ring $\mathbb Z\mathbb P$ is a domain.
We take an ambient field $\mathcal F$  to be the field of rational functions in $n$ independent variables with coefficients in $\mathbb Z\mathbb P$.

An integer matrix $B=(b_{ij})$ is called {\bf skew-symmetrizable} if there exists a diagonal matrix $S$ with positive integer diagonal entries such that $SB$ is skew-symmetric. Let $B$ be  a skew-symmetrizable matrix , we can encode the sign pattern of matrix entries of $B$ by the directed graph $\Gamma(B)$ with the vertices $1,2,\cdots,n$ and the directed edges $(i,j)$ for $b_{ij}>0$. A skew-symmetrizable matrix matrix $B$ is called {\bf acyclic} if $\Gamma(B)$ has no oriented cycles.

\begin{Definition}
A {\bf(labeled) seed} $\Sigma$ in $\mathcal F$ is a triplet $(X,Y,B)$ such that

(i)  $X=(x_1,\cdots, x_n)$, where $x_1\cdots,x_n$  form a free generating set of $\mathcal F$. $X$ is called a {\bf cluster} and $x_1\cdots,x_n$  are called {\bf cluster variables}.

(ii) $Y=(y_1,\cdots,y_n)$ is an $n$-tuple of elements in $\mathbb P$, where $y_1,\cdots,y_n$ are called {\bf coefficients}.

(iii) $B=(b_{ij})$ is a $n\times n$ skew-symmetrizable  integer matrix, called an {\bf exchange matrix}.

The seed $\Sigma$ is called {\bf acyclic} if $B$ is acyclic.
\end{Definition}

Denote by $[a]_+=max\{a,0\}$ for any $a\in \mathbb R$.
\begin{Definition}
Let $\Sigma=(X,Y,B)$ be a seed in $\mathcal F$. Define the {\bf mutation}  of the seed $\Sigma$ in the direction $k\in\{1,\cdots,n\}$ as a new triple $\mu_k(\Sigma)=\bar \Sigma=(\bar X, \bar Y, \bar B)$ in $\mathcal F$:
\begin{eqnarray}
\label{eq1}\bar x_i&=&\begin{cases}x_i~,&\text{if } i\neq k\\ \frac{y_k\prod\limits_{j=1}^nx_j^{[b_{jk}]_+}+ \prod\limits_{j=1}^nx_j^{[-b_{jk}]_+}}{({1\bigoplus y_k})x_k},~& \text{if }i=k.\end{cases}\\
\bar y_i&=&\begin{cases} y_k^{-1}~,& i=k\\ y_iy_k^{[b_{ki}]_+}(1\bigoplus y_k)^{-b_{ki}}~,&otherwise.
\end{cases}\\
\bar b_{ij}&=&\begin{cases}-b_{ij}~,& i=k\text{ or } j=k;\\ b_{ij}+b_{ik}[-b_{kj}]_++[b_{ik}]_+b_{kj}~,&otherwise.\end{cases}.
\end{eqnarray}

\end{Definition}
It can be seen that $\mu_k(\Sigma)$ is also a seed and $\mu_k(\mu_k(\Sigma))=\Sigma$.

\begin{Definition}
 (i) A {\bf cluster pattern}  $M$ in $\mathbb P$ is an assignment of a seed  $\Sigma_t$ to every vertex $t$ of the $n$-regular tree $\mathbb T_n$, such that for any edge $t^{~\underline{\quad k \quad}}~ t^{\prime},~\Sigma_{t^{\prime}}=\mu_k(\Sigma_t)$. If $\mathbb P=Trop(u_i: i\in I)$ with $|I|<+\infty$, $M$ is called a {\bf cluster pattern of geometric type}.

 (ii) A cluster pattern $M$ in $Trop(y_1,\cdots,y_n)$ is said to be a {\bf principal coefficients cluster pattern} if $M$ has a seed $\Sigma_{t_0}=(X_{t_0},Y_{t_0},B_{t_0})$ such that $Y_{t_0}=(y_1,\cdots,y_n)$. In this case, $M$ is called  a {\bf cluster pattern with principal coefficients} at $t_0$.
\end{Definition}
Note that a cluster pattern $M$ is uniquely determined by any seed of $M$.
We always denote $\Sigma_t=(X_t,Y_t,B_t)$, where $X_t=(x_{1;t},\cdots, x_{n;t}),~ Y_t=(y_{1;t},\cdots, y_{n;t}), ~B_t=(b_{ij}^t).$
\begin{Definition} (a)~Let $M$ be a cluster pattern,   the {\bf cluster algebra} $\mathcal A(M)$ associated with the given cluster pattern $M$ is the $\mathbb {ZP}$-subalgebra of the field $\mathcal F$ generated by all cluster variables of $M$. If $M$ is a cluster pattern of geometric type, then $\mathcal A(M)$ is called a {\bf cluster algebra of geometric type}.

(b)~If $M$ is a cluster pattern with principal coefficients at $t_0$, then the cluster algebra $\mathcal A(M)$ is called a {\bf cluster algebra with principal coefficients} at $t_0$.

 \end{Definition}

\begin{Theorem}\label{laurent}
Let $\mathcal A(M)$ be a cluster algebra, and $x_{i;t}$ be a cluster variable of $\mathcal A(M)$. Let  $X_{t_0}$ be any cluster of $\mathcal A(M)$, then

(i) (Laurent Phenomenon \cite{GHKK}) $x_{i;t}$ can be expressed a Laurent polynomial in $x_{1;t_0},\cdots, x_{n;t_0}$ with coefficients in $\mathbb {ZP}$.

(ii) (Positivity of cluster variables \cite{FZ}) if further, $\mathcal A(M)$ is a {\em skew-symmetrizable}  cluster algebra of geometry type, then $x_{i;t}$ can be expressed a Laurent polynomial in $x_{1;t_0},\cdots, x_{n;t_0}$ with coefficients in $\mathbb {NP}$.

\end{Theorem}

\begin{Remark}\label{remklaurent}
By Theorem \ref{laurent} (ii) and Fomin-Zelevinsky's separation of addition formula, i.e., Theorem 3.7 of \cite{FZ3}, we know positivity of cluster variables holds for any skew-symmetrizable cluster algebra.
\end{Remark}

\subsection{GLP patterns and  GLP algebras}

Let $\mathbb P$ be an Abelian group such that its group ring $\mathbb Z\mathbb P$ is a domain.
We take an ambient field $\mathcal F$  to be the field of rational functions in $n$ independent variables with coefficients in $\mathbb {ZP}$.

Recall that $X_t=\{x_{1;t},\cdots,x_{n;t}\}$ is called a {\bf cluster} in $\mathcal F$ if $x_{1;t},\cdots,x_{n;t}$ form a free generating set of $\mathcal F$. And in this case, $x_{1;t},\cdots,x_{n;t}$ are called {\bf cluster variables}.

 Denote $\mathcal L(t)=\mathbb {ZP}[x_{1;t}^{\pm1},\cdots,x_{n;t}^{\pm1}],\; \mathcal L^+(t)=\mathbb {NP}[x_{1;t}^{\pm1},\cdots,x_{n;t}^{\pm1}]$.

\begin{Definition} \label{GLPA} Assume $T$ is a (finite or infinite) index set.

(i) Let $M_T:=\{X_t|t\in T\}$ be a family of clusters in $\mathcal F$, which  is called a {\bf generalized Laurent phenomenon pattern  of rank $n$} or shortly {\bf GLP pattern } if $X_{t^{\prime}}\subseteq \mathcal L(t)$ for any $t,t^{\prime}\in T$. In this case,  $X_t$ is  called a {\bf cluster} of $M_T$ at $t$, whose elements are called {\bf cluster variables} of $M_T$.

(ii) Given a GLP pattern $M_T$, the $\mathbb {ZP}$-subalgebra $\mathcal A(M_T)$ of $\mathcal F$ generated by all cluster variables of $M_T$, is called the {\bf generalized Laurent phenomenon algebra  of rank $n$ } or shortly {\bf GLP algebra} associated with $M_T$.

(iii) A GLP pattern $M_T$ is said to be {\bf positive}  if $X_{t^{\prime}}\subseteq \mathcal L^+(t)$ for any $t,t^{\prime}\in T$.  In this case,  $\mathcal A(M_T)$ is called a {\bf positive GLP algebra},  or say, the {\bf positivity} of the GLP algebra $\mathcal A(M_T)$ holds.
\end{Definition}

It is easy to see that
cluster algebras, generalized cluster algebras \cite{CS}, Laurent phenomenon algebras \cite{LP} are examples of GLP algebras.
 In Definition \ref{GLPA}, the word ``generalized" is just to distinguish Lam and Pylyavskyy's Laurent phenomenon algebras in \cite{LP}.

Note that if a GLP algebra $\mathcal A(M_T)$ is positive, we sometimes also say that the positivity of cluster variables holds for $\mathcal A(M_T)$.
 By Remark \ref{remklaurent}, we know {\em skew-symmetrizable}  cluster algebras are examples of positive GLP algebras.

From the definition of (positive) GLP pattern, each single cluster variable forms a trivial positive GLP pattern. Then, giving the partial order in the set of GLP patterns via the inclusion relation and using the Zorn's Lemma, it is easy to obtain the following:
\begin{Proposition}\label{maxsub}
Let $M_T$ be GLP pattern.

(i) For every nonempty subset $T^{\prime}$ of $T$, $M_{T^{\prime}}=\{X_t|t\in T^{\prime}\}$ is a GLP pattern. We call $M_{T^{\prime}}$  a {\bf subpattern} of $M_T$.

(ii) There exists a {\em maximal} positive subpattern $M_{T^{\prime}}$ of $M_T$.

\end{Proposition}

Denote by ${\bf x}_t^{\bf a}=\prod\limits_{i=1}^{n}x_{i;t}^{a_i}\in\mathcal F$ with ${\bf a}=(a_1,\cdots,a_n)^{\top}\in\mathbb Z^n$, which
   is called a {\bf cluster monomial} in $X_t$ if ${\bf a}\in\mathbb N^n$ and is called a {\bf proper Laurent monomial} in $X_t$ if ${\bf a}\notin\mathbb N^n$. Denote by $CM(t)$  the set of cluster monomials in $X_t$. It is easy to see that if $X_{t_1}$ and $X_{t_2}$ have common cluster variables, then $CM(t_1)\cap CM(t_2)\not=\emptyset$.

\begin{Definition}(\cite{CLF,CKLP})\label{deflaurent}
(i) A GLP algebra $\mathcal A(M_T)$  is said to have  the {\bf proper Laurent monomial property with respect to a cluster $X_{t_0}$} if for any cluster $X_t$ of $\mathcal A(M_T)$, every cluster monomial ${\bf x}_{t}^{a}\in CM(t)\setminus CM(t_0)$ is a $\mathbb {ZP}$-linear combination of proper Laurent monomials in $X_{t_0}$.

(ii) A GLP algebra $\mathcal A(M_T)$  is said to have  the {\bf proper Laurent monomial property} if $\mathcal A(M_T)$ has the proper Laurent monomial property with respect to each cluster of $M_T$.
\end{Definition}

 The linear independence of cluster monomials of cluster algebras was proved in the skew-symmetric case from triangulated surfaces and the   skew-symmetrizable  case respectively in (Theorem 6.4, \cite{CLF}) and (Theorem 7.20, \cite{GHKK}) via the proper Laurent monomial property.
 The following theorem generalizes  Theorem 6.4 of [5] to the case of GLP
algebra $\mathcal A(M_T)$  in the similar method. For convenience of readers, we repeat the proof here.

\begin{Theorem}\label{properthm}
Let $\mathcal A(M_T)$ be a GLP algebra with  the proper Laurent monomial property. Then cluster monomials of $\mathcal A(M_T)$ are linearly independent
over $\mathbb {ZP}$.
\end{Theorem}
\begin{proof}
Suppose that $\sum\limits_{t, {\bf v}} c_{t,{\bf v}}{\bf x}_t^{\bf v}=0$, where ${\bf v}\in\mathbb N^n, c_{t,{\bf v}}\in\mathbb {ZP}$. Fix a cluster $X_{t_0}$ and a ${\bf v}_0\in \mathbb N^n$, by the proper Laurent monomial property,  each ${\bf x}_t^{\bf v}\notin CM(t_0)$ is a sum of proper Laurent monomials in $\mathcal L(t_0)$. Thus
$\sum\limits_{t, {\bf v}} c_{t,{\bf v}}{\bf x}_t^{\bf v}$ can be written as a Laurent polynomial in $\mathcal L(t_0)$ with the form of $\Sigma_1+\Sigma_2$,
where $\Sigma_1$ is a sum of monomials in $\mathcal L(t_0)$ and $\Sigma_2$ is a sum of proper Laurent monomials in $\mathcal L(t_0)$.
The coefficient of ${\bf x}_{t_0}^{{\bf v}_0}$  in $\Sigma_1$, as well as in $\Sigma_1+\Sigma_2$, is precisely $c_{t_0,{\bf v}_0}$.
  Then, $\Sigma_1+\Sigma_2=0$ results in $c_{t_0,{\bf v}_0}=0$.
 And thus cluster monomials of $\mathcal A(M_T)$ are linearly independent
over $\mathbb {ZP}$.
\end{proof}

 The above theorem tells us the significance of the proper Laurent monomial property to characterize the linearly independence of cluster monomials for GLP algebras.
 In the next section, we will give the sufficient condition for a GLP algebra to have the proper Laurent monomial property.
 Moreover, the result in Section 4 will show that skew-symmetric cluster algebras, as a special kind of GLP algebras, always satisfy this sufficient condition and then they have the proper Laurent monomial property.

\section {From {\bf d}-vector-positivity to proper Laurent monomial  property of GLP algebras}

 From Conjecture \ref{conjd} and its partial answers in  \cite{FZ}, \cite{FST} and \cite{GHKK}, we know the importance of the positivity of {\bf d}-vectors of cluster variables in a cluster algebra. Motivated by this view, we now give the notion of {\bf d}-vector-positivity for general GLP algebras and discuss what there would happen under this condition.

Given a GLP algebra $\mathcal A(M_T)$, let $X_{t_0}$ be a cluster of $\mathcal A(M_T)$. By the definition of GLP algebra,  any cluster variable $x_{i;t}$ of $\mathcal A(M_T)$ has the form of $x_{i;t}=\sum\limits_{{\bf v}\in V} c_{\bf v}{\bf x}_{t_0}^{\bf v}$, where $V$ is a subset of $\mathbb Z^n$, $0\neq c_{\bf v}\in \mathbb {ZP}$. Let $-d_{ji}^t$ be the minimal exponent of $x_{j;t_0}$ appearing in the expansion $x_{i;t}=\sum\limits_{{\bf v}\in V} c_{\bf v}{\bf x}_{t_0}^{\bf v}$. Then $x_{i;t}$ has the form of
\begin{eqnarray}
\label{eqd}x_{i;t}=\frac{f_{i;t}^{t_0}(x_{1;t_0},\cdots,x_{n,t_0})}{x_{1;t_0}^{d_{1i}^t}\cdots x_{n;t_0}^{d_{ni}^t}},
\end{eqnarray}
where $f_{i;t}^{t_0}\in\mathbb {ZP}[x_{1;t_0},\cdots,x_{n;t_0}]$ with  $x_{j;t_0}\nmid f_{i;t}$.
The vector ${\bf d}^{t_0}(x_{i;t}) = (d_{1i}^{t},\cdots, d_{ni}^t)^{\top}$ is called the {\bf denominator vector} (briefly, {\bf d-vector})  of the cluster variable $x_{i;t}$ with respect to $X_{t_0}$, and $D_t^{t_0}=({\bf d}^{t_0}(x_{1;t}) \cdots {\bf d}^{t_0}(x_{n;t}))=({d_{ij}^{t}})_{n\times n}$ is called the {\bf D-matrix} of the cluster $X_t$ with respect to $X_{t_0}$.
Clearly, $D_{t_0}^{t_0}=-I_n$.
Sometimes we  write ${\bf d}^{t_0}(x_{i;t})$ as the {\em {\bf d}-vector} of $x_{i;t}$ with respect to $X_{t_0}$.

\begin{Definition}\label{defd}
Let $\mathcal A(M_T)$ be a GLP algebra.

(i) A cluster variable $x$ of $\mathcal A(M_T)$ is called {\bf d-vector-positive}, if  $x\in X_{t}$  or  ${\bf d}^{t}(x)\in \mathbb N^n$ holds for any cluster $X_t$ of  $\mathcal A(M_T)$.

(ii) A cluster $X_{t_0}$ of $\mathcal A(M_T)$ is called {\bf d-vector-positive}, if any cluster variable  in $X_{t_0}$ is  d-vector-positive.

(iii) $\mathcal A(M_T)$ is called {\bf d-vector-positive}, if any cluster $X_t$ of  $\mathcal A(M_T)$  is  {\bf d}-vector-positive.

\end{Definition}

Using this definition,  Conjecture \ref{conjd} can be described as that every cluster algebra is {\bf d}-vector-positive.
 Then as mentioned in Section 1,  cluster algebras are  {\bf d}-vector-positive in the cases that (i)\;
 cluster algebras of rank $2$ (Theorem 6.1 of \cite{FZ}), (ii)\; cluster algebras arising from surfaces (see \cite{FST}), and (iii)\; cluster algebras of finite type (see \cite{CCS,CP}).

In this section, we will prove that if  a GLP algebra $\mathcal A(M_T)$ is both positive and {\bf d}-vector-positive, then $\mathcal A(M_T)$ has the {\em  proper Laurent monomial property} (Theorem \ref{mainthm}).

Given a GLP pattern $M_T$, we can select a subset $T^{\prime}$ of $T$ such that $M_{T^{\prime}}$ is a {\em {\bf d}-vector-positive} GLP subpattern of $M_T$. For example, fix an element $t_0$ of $T$, and let $T^{\prime}=\{t_0\}$, then $M_{T^{\prime}}$ is a {\em {\bf d}-vector-positive} GLP subpattern of $M_T$. Similarly with Proposition \ref{maxsub}, it is easy to see  by Zorn' lemma that  every GLP pattern has a maximal  {\bf d}-vector-positive GLP subpattern.

\begin{Lemma}\label{mainlemma}
Let $\mathcal A(M_T)$ be a positive GLP algebra and $X_{t_0}$ be a  {\bf d}-vector-positive cluster of $\mathcal A(M_T)$. Then $\mathcal A(M_T)$ has the proper Laurent monomial property with respect to $X_{t_0}$.
\end{Lemma}
\begin{proof}
Let $X_t$ be a cluster of $\mathcal A(M_T)$, and ${\bf x}_t^{\bf a}=\prod\limits_{i=1}^{n}x_{i;t}^{a_i}$ be a cluster monomial in $CM(t)\setminus CM(t_0)$, i.e.,  there exists $a_k>0$ such that $x_{k;t}$
is not a cluster variable in $X_{t_0}$. Since $\mathcal A(M_T)$ is a positive GLP algebra, ${\bf x}_t^{\bf a}\in\mathcal L^+(t_0)$ can be expressed as
\begin{eqnarray}
{\bf x}_t^{\bf a}=\sum\limits_{{\bf v}\in V} c_{\bf v}{\bf x}_{t_0}^{\bf v},\nonumber
 \end{eqnarray}
 where $V$ is a finite subset of $\mathbb Z^n$, and $0\neq c_{\bf v}\in\mathbb {NP}, {\bf v}\in V$.
Also because $\mathcal A(M_T)$ is a positive GLP algebra, there exists polynomial $f_1,\cdots,f_n\in \mathbb {NP}[x_{1;t},\cdots,x_{n;t}]$ with $x_{j;t}\nmid f_i$ such that
 $$x_{i;t_0}=\frac{f_i(x_{1;t},\cdots,x_{n;t})}{{\bf x}_t^{{\bf d}_{i;t_0}^{t}}}.$$
Denote by $F^{\bf v}=f_1^{v_1}\cdots f_n^{v_n}$.  Since $v_1,\cdots,v_n\in \mathbb Z$, $F^{\bf v}$ can be written in  the form of $F^{\bf v}=\frac{F_{1;{\bf v}}}{F_{2;{\bf v}}}$, where
$F_{1;{\bf v}}, F_{2;{\bf v}}\in \mathbb {NP}[x_{1;t},\cdots,x_{n;t}]$  with $x_{j;t}\nmid F_{1;{\bf v}}$ and $x_{j;t}\nmid F_{2;{\bf v}}$, $j=1,2,\cdots,n$.
 Thus, $${\bf x}_t^{\bf a}=\sum\limits_{{\bf v}\in V} c_{\bf v}{\bf x}_{t_0}^{\bf v}=
\sum\limits_{{\bf v}\in V} c_{\bf v}{\bf x}_{t}^{-D_{t_0}^{t}\bf v}F^{\bf v}=\sum\limits_{{\bf v}\in V} c_{\bf v}{\bf x}_{t}^{-D_{t_0}^{t}\bf v}\frac{F_{1;{\bf v}}}{F_{2;{\bf v}}}.$$
 From the above equality, we obtain an equality as the following form:
 $$
{\bf x}_t^{\bf a}g(x_{1;t},\cdots,x_{n;t})=\sum\limits_{{\bf v}\in V} c_{\bf v}{\bf x}_{t}^{-D_{t_0}^{t}{\bf v}}g_{\bf v}(x_{1;t},\cdots,x_{n;t}),
$$
which can be written as
$$g(x_{1;t},\cdots,x_{n;t})=\sum\limits_{{\bf v}\in V} c_{\bf v}{\bf x}_{t}^{-D_{t_0}^{t}{\bf v}-{\bf a}}g_{\bf v}(x_{1;t},\cdots,x_{n;t}),$$
where $c_{\bf v}\in\mathbb {NP}$, and $g,~g_{\bf v}\in \mathbb {NP}[x_{1;t},\cdots,x_{n;t}]$ with $x_{j;t}\nmid g$ and $x_{j;t}\nmid g_{\bf v}$, $j=1,2,\cdots,n$.

Thus we obtain $-D_{t_0}^{t}{\bf v}-{\bf a}\in\mathbb N^n$. So the $k$-th component of  $D_{t_0}^{t}{\bf v}+{\bf a}$ satisfying $(d_{k1}^{t_0}v_1+\cdots+d_{kn}^{t_0}v_n)+a_k\leq0$, where $d_{kj}^{t_0}$ is the $k$-th component of ${\bf d}_{j;t_0}^t$, $j=1,2,\cdots,n$.
If $d_{kj}^{t_0}<0$, then ${\bf d}_{j;t_0}^t\notin \mathbb N^n$. Then, since $X_{t_0}$ is a  {\bf d}-vector-positive cluster of $\mathcal A(M_T)$, we obtain $x_{j;t_0}\in X_t$ and more precisely, $x_{j;t_0}=x_{k;t}$. This contradicts $x_{k;t}\notin X_{t_0}$.
So $d_{kj}^{t_0}\geq0$ for $j=1,2,\cdots,n$. Then by $a_k>0$ and $(d_{k1}^{t_0}v_1+\cdots+d_{kn}^{t_0}v_n)+a_k\leq0$, we can obtain ${\bf v}\notin \mathbb N^n, {\bf v}\in V$.
\end{proof}

From  Lemma \ref{mainlemma} and Theorem \ref{properthm}, it follows that
\begin{Theorem}\label{mainthm}
   Each positive and  {\bf d}-vector-positive GLP algebra $\mathcal A(M_T)$ has the proper  Laurent property and thus its cluster monomials are linearly independent.
\end{Theorem}

\section{ Conjecture \ref{conjd} for skew-symmetric cluster algebras}

In this section, we  give an affirmative answer to  Conjecture \ref{conjd} for any skew-symmetric cluster algebras (Theorem \ref{thmdgood}).  This means Theorem \ref{mainthm} is always satisfied by such cluster algebras.

Let $\mathcal A(M)$ be a cluster algebra and $z$ be its cluster variable. Denote by $I(z)$  the set of clusters $X_{t_0}$ of $\mathcal A(M)$ such that $z$ is a cluster variable in $X_{t_0}$. For two vertices $t_1,t_2$ of $\mathbb T_n$, let $l(t_1,t_2)$ be the distance between $t_1$ and $t_2$ in the $n$-regular tree $\mathbb T_n$.  For a cluster $X_t$ of $\mathcal A(M)$, define the distance  between $z$ and $X_t$ as $dist(z,X_t):=min\{l(t_0,t):t_0\in I(z)\}$.

The following theorem is from \cite{LS}:
 \begin{Theorem}(Theorem 4.1, Proposition 5.1 and 5.5 in \cite{LS})\label{Lee}

 Let $\mathcal A(M_{\mathbb T_n})$ be a skew-symmetric cluster algebra of geometric
type with a cluster $X_t$. Let $z$ be a cluster variable of $\mathcal A(M)$, and $X_{t_0}$ be a cluster containing $z$ such that $dist(z,X_t)=l(t_0,t)$. Let $\sigma$ be the unique
sequence in $\mathbb T_n$ relating the seeds at $t_0$ and $t$ in which $p,q$ denote the last two directions (clearly, $p\neq q$):
$$\sigma:\;\; t_0^{~\underline{  \quad k_1\quad   }}~ t_1^{~\underline{\quad k_2 \quad}}  ~\cdots ~t_{m-2}^{~\underline{\quad p=k_{m-1} \quad}}~t_{m-1}=u ^{~\underline{~\quad q=k_{m} \quad}}~ t_m=t.$$
For $e\neq p,q$, let $u^{~\underline{  \quad e \quad   }}~ v$ and $t^{~\underline{  \quad e \quad   }}~ w$ be the edges of $\mathbb T_n$ in the same direction $e$. Then,

(i) $z$ can be written as $z=P_t+Q_t$, with $P_t\in L_1:=\mathbb {NP}[x_{q;u},x_{p;t}^{\pm1};(x_{i;t}^{\pm1})_{i\neq p,q}]$, and $Q_t\in L_2:=\mathbb {NP}[x_{q;t},x_{p;t}^{\pm1};(x_{i;t}^{\pm1})_{i\neq p,q}]$.  Moreover, $P_t$ and $Q_t$ are unique up to $L_1\cap L_2$.

(ii)  There exists a Laurent monomial $F$ appearing in the expansion $z=P_t+Q_t$  such that the variable $x_{e;t},$  has nonnegative exponent in $F$.

(iii)   There exist $P_{1;t}\in L_3,P_{2;t}\in L_4,Q_{1;t}\in L_5,Q_{2;t}\in L_6$ such that  $P_{1;t}+P_{2;t}=P_t$ and $Q_{1;t}+Q_{2;t}=Q_t$, where
\begin{eqnarray}
 L_3&:=&\mathbb {NP}[x_{q;u},x_{e;v};(x_{i;t}^{\pm1})_{i\neq p,q}],~~~~L_4:=\mathbb {NP}[x_{q;u},x_{e;u};(x_{i;t}^{\pm1})_{i\neq p,q}],\nonumber\\
L_5&:=&\mathbb {NP}[x_{q;t},x_{e;t};(x_{i;t}^{\pm1})_{i\neq p,q}],~~~~L_6:=\mathbb {NP}[x_{q;t},x_{e;w};(x_{i;t}^{\pm1})_{i\neq p,q}].\nonumber
\end{eqnarray}
Thus $z$ has the form of $z=P_{1;t}+P_{2;t}+Q_{1;t}+Q_{2;t}$, where $P_{1;t},P_{2;t},Q_{1;t},Q_{2;t}$ are unique up to $L_3\cap L_4\cap L_5\cap L_6$.
 \end{Theorem}
\begin{Lemma}\label{lempq} Keep the notations in Theorem \ref{Lee}.
There exists a Laurent monomial in Laurent expansion of $z$ with respect to the seed at $t$  in which  $x_{e;t}$ appears with non-negative exponent for $e\neq p,q$.
\end{Lemma}
\begin{proof}
By Theorem \ref{Lee} (ii), there exists a Laurent monomial $F$ appearing in the expansion $z=P_t+Q_t$  such that the variable $x_{e;t}$  has non-negative exponent in $F$.
We know $$x_{q;u}=\frac{y_{q;t}\prod\limits_{j=1}^nx_{j;t}^{[b_{jq}^t]_+}+ \prod\limits_{j=1}^nx_{j;t}^{[-b_{jk}^t]_+}}{({1\bigoplus y_{q;t}})x_{q;t}}.$$
Substituting the above equality into $z=P_t+Q_t$, then we obtain the Laurent expansion of $z$ with respect to $X_t$. It is easy to see that in this Laurent expansion there exists a Laurent monomial such that  $x_{e;t}$ appear in this Laurent monomial with nonnegative exponent.
\end{proof}

Now, we give the main result in this section.
 \begin{Theorem}\label{thmdgood}
Skew-symmetric cluster algebras are  {\bf d}-vector-positive.
 \end{Theorem}

\begin{proof}
From the exchange relation of $D$-matrices (see (7.7) of \cite{FZ3}), we know the notion of {\bf d}-vectors is independent of the choice of coefficient system. Hence, we can assume that  $\mathcal A(M_{\mathbb T_n})$ is a skew-symmetric cluster algebra of geometric type.

Let $X_{t^{\prime}}$ be a cluster of $\mathcal A(M_{\mathbb T_n})$, $z$ be a cluster
variable with $z\notin X_{t^{\prime}}$, then $dist(z,X_{t^{\prime}})=:m+1>0$. We will show the {\bf d}-vector ${\bf d}^{t^{\prime}}(z)$ of $z$ with respect to $X_{t^{\prime}}$ is in $\mathbb N^n$.

Choose $X_{t_0}\in I(z)$ such that $l(t_0,t^{\prime})=dist(z,X_{t^{\prime}})=m+1>0$.  Since $\mathbb T_n$ is connected as a tree, there is  the unique sequence $\sigma$ linking $t_0$ and $t^{\prime}$ in $\mathbb T_n$:
$$t_0^{~\underline{  \quad k_1\quad   }}~ t_1^{~\underline{\quad k_2 \quad}}  ~\cdots ~t_{m-2}^{~\underline{\quad p=k_{m-1} \quad}}~t_{m-1}=u ^{~\underline{~\quad q=k_{m} \quad}}~ t_m=t~^{~\underline{\quad k_{m+1}(\neq q) \quad}}~t_{m+1}=t^{\prime}.$$
By the choice of $X_{t_0}$, we know $dist(z,X_{t_j})=l(t_0,t_j)=j$ for $j=1,2,\cdots,m+1$.

We prove  ${\bf d}^{t^{\prime}}(z)\in\mathbb N^n$ by induction on $dist(z,X_{t^{\prime}})=m+1>0$.

Clearly,  ${\bf d}^{t^{\prime}}(z)\in\mathbb N^n$ if $dist(z,X_{t^{\prime}})=1,2$.

As {\em Inductive Assumption}, we assume that ${\bf d}^{t^{\star}}(z)\in\mathbb N^n$ for any $t^{\star}$ such that $1\leq dist(z,X_{t^{\star}})<dist(z, X_{t^{\prime}})$.

Because $1\leq dist(z,X_t)<dist(z, X_{t^{\prime}})$, it means ${\bf d}^{t}(z)\in\mathbb N^n$. Since  $X_{t^{\prime}}=\mu_{k_{m+1}}(X_t)$, Proposition 1.5 of \cite{RS} says that the $i$-th component of ${\bf d}^{t}(z)$ and the $i$-th component of ${\bf d}^{t^{\prime}}(z)$ are equal for $i\neq k_{m+1}$. Thus, to show ${\bf d}^{t^{\prime}}(z)\in\mathbb N^n$, we only need to show that the $k_{m+1}$-th component of ${\bf d}^{t^{\prime}}(z)$ is nonnegative.
We prove this in two cases,that is, Case (I): $k_{m+1}\neq p$ and Case (II): $k_{m+1}=p$.

The proof of Case (I):

The Laurent expansion of $z$ with respect to $X_{t^{\prime}}$ has the form of $z=\sum\limits_{{\bf v}\in V}c_{\bf v}{\bf x}_{t^{\prime}}^{\bf v}$, where $V$ is a subset of $\mathbb Z^n$, $0\neq c_{\bf v}\in \mathbb {NP}$. If $k_{m+1}$-th component of ${\bf d}^{t^{\prime}}(z)$ is negative ($k_{m+1}\neq p,q$), then  the exponent of $x_{k_{m+1};t^{\prime}}$ in each ${\bf x}_{t^{\prime}}^{\bf v}$ must be positive, ${\bf v}\in V$. We know $$x_{k_{m+1};t^{\prime}}=\frac{y_{k_{m+1};t^{\prime}}\prod\limits_{j=1}^nx_{j;t^{\prime}}^{[b_{jk_{m+1}}^{t^{\prime}}]_+}+ \prod\limits_{j=1}^nx_{j;t^{\prime}}^{[-b_{jk_{m+1}}^{t^{\prime}}]_+}}{({1\bigoplus y_{k_{m+1};t^{\prime}}})x_{k_{m+1};t}}.$$
Substituting the above equation into $z=\sum\limits_{{\bf v}\in V}c_{\bf v}{\bf x}_{t^{\prime}}^{\bf v}$, we can obtain the Laurent expansion of $z$ with respect to $X_t$. And the exponent of $x_{k_{m+1};t}$ in each Laurent monomial appearing in the obtained  Laurent expansion is negative. This contradicts Lemma \ref{lempq}.
Thus if $k_{m+1}\neq p,q$, then $k_{m+1}$-th component of ${\bf d}^{t^{\prime}}(z)$ is nonnegative and we have  ${\bf d}^{t^{\prime}}(z)$ is in $\mathbb N^n$.

The preparation for the proof of Case (II):

Consider the maximal rank two mutation subsequence in directions $p,q$ at the end of $\sigma$. This sequence connects $t^{\prime}$ to a vertex $t_r$.  Thus we have the following two casesㄩ
\begin{eqnarray}\label{eqnseq1}
t_0\cdots \cdot~t_{r-1}^{~\underline{  \quad k_r(\neq q)\quad   }}~ t_r^{~\underline{\quad p \quad}}~t_{r+1}^{~\underline{\quad q \quad}}~t_{r+2}^{~\underline{\quad p \quad}} ~\cdots ~t_{m-2}^{~\underline{\quad p \quad}}~t_{m-1}^{~\underline{\quad q \quad}}~t_m^{~\underline{~\quad p \quad}}~t_{m+1}=t^{\prime}.
\end{eqnarray}
 or
\begin{eqnarray}\label{eqnseq2}
t_0\cdots \cdot~t_{r-1}^{~\underline{  \quad k_r(\neq p)\quad   }}~ t_r^{~\underline{\quad q \quad}}~t_{r+1}^{~\underline{\quad p \quad}}~t_{r+2}^{~\underline{\quad q \quad}}~\cdots ~t_{m-2}^{~\underline{\quad p \quad}}~t_{m-1}^{~\underline{\quad q \quad}}~t_m^{~\underline{~\quad p \quad}}~t_{m+1}=t^{\prime}.
\end{eqnarray}

Let $t_r^{~\underline{\quad q \quad}}~t_{a}$  in the first case and let $t_r^{~\underline{\quad p \quad}}~t_{b}$  in the second case. It is easy to see that $r\leq m-2$ in the first case and $r\leq m-3$ in the second case.

\begin{Lemma}\label{lempositive}
(i) In the first case $t_r^{~\underline{\quad q \quad}}~t_{a}$,  the $p$-th components of {\bf d}-vectors of $x_{p;t_a},x_{q;t_a},x_{p;t_r},x_{p;t_{r+1}}$, $x_{q;t_{r+1}},x_{q;t_{r+2}}$ with respect to $X_{t^{\prime}}$
are nonnegative.

(ii) In the second case $t_r^{~\underline{\quad p \quad}}~t_{b}$,  the $p$-th components of {\bf d}-vectors of $x_{p;t_b},x_{q;t_b},x_{q;t_r},x_{p;t_{r+1}},$ $x_{q;t_{r+1}}, x_{p;t_{r+2}}$ with respect to $X_{t^{\prime}}$
are nonnegative.
\end{Lemma}
\begin{proof}
(i) If the $p$-th components of  {\em {\bf d}-vector} ${\bf d}^{t^{\prime}}(x)$ of some cluster variable $x\in\{x_{p;t_a},x_{q;t_a},x_{p;t_r},$ $x_{p;t_{r+1}},x_{q;t_{r+1}},x_{q;t_{r+2}}\}$  with respect to $X_{t^{\prime}}$ is negative. Say $x\in X_{t_E}, E\in\{a,r,r+1,r+2\}$.
Since the clusters  $X_{t_a},X_{t_r},X_{t_{r+1}}$, $X_{t_{r+2}}$  can be obtained from the cluster $X_{t^{\prime}}$  by  sequences of mutations using only  two directions $p$ and $q$. Then by Theorem 6.6 in \cite{FZ}, we know $x=x_{p;t^{\prime}}$. Then, $X_{t_E}$ and $X_{t^{\prime}}$ have at least $n-1$ common cluster variables. By Theorem 5 of \cite{GSV} or Theorem 4.23 (c) of \cite{CL}, we know $X_{t^{\prime}}=X_{t_E}$ or $X_{t^{\prime}}=\mu_p(X_{t_E})$ or $X_{t^{\prime}}=\mu_q(X_{t_E})$ as sets.
Thus
\begin{eqnarray}\label{eqnless}
m+1=dist(z,X_{t^{\prime}})\leq max\{dist(z,X_{t_E}),dist(z,\mu_p(X_{t_E}),dist(z,\mu_q(X_{t_E})\}.
\end{eqnarray}

 On the other hand, by the sequence (\ref{eqnseq1}) and $t_r^{~\underline{\quad q \quad}}~t_{a}$, we know $$max\{dist(z,X_{t_E}),dist(z,\mu_p(X_{t_E}),dist(z,\mu_q(X_{t_E})\}\leq r+3\leq m+1.$$
Thus we obtain $max\{dist(z,X_{t_E}),dist(z,\mu_p(X_{t_E}),dist(z,\mu_q(X_{t_E})\}=m+1$, and this will result in that $r=m-2$ and $E=m$, i.e. $X_{t_E}=X_{t_m}$. But this  contradicts that $x_{p;t^{\prime}}=x\in X_{t_E}=X_{t_m}$. So  the $p$-th components of {\bf d}-vectors of $x_{p;t_a},x_{q;t_a},x_{p;t_r},x_{p;t_{r+1}}$, $x_{q;t_{r+1}},x_{q;t_{r+2}}$ with respect to $X_t^{\prime}$ are nonnegative.

(ii) By the same argument with (i), we have
\begin{eqnarray}\label{mmmm}
m+1=dist(z,X_{t^{\prime}})\leq max\{dist(z,X_{t_E}),dist(z,\mu_p(X_{t_E}),dist(z,\mu_q(X_{t_E})\},
\end{eqnarray}
 and, on the other hand,
$$max\{dist(z,X_{t_E}),dist(z,\mu_p(X_{t_E}),dist(z,\mu_q(X_{t_E})\}\leq r+3\leq m,$$
 which contradicts to (\ref{mmmm}). So, the $p$-th components of {\bf d}-vectors of $x_{p;t_b},x_{q;t_b},x_{q;t_r},x_{p;t_{r+1}}$, $x_{q;t_{r+1}}, x_{p;t_{r+2}}$ with respect to $X_{t^{\prime}}$ are nonnegative.
\end{proof}

{\em Return to the proof of Theorem \ref{thmdgood}.}
The proof of case (II):

 Applying Theorem \ref{Lee} (iii) at the vertex $t_{r+1}$ with respect to the directions $q$ (in case of (\ref{eqnseq1})) and $p$ (in case of (\ref{eqnseq2})), we get that either
$$z\in\mathbb{NP}[x_{p;t_a},x_{q;t_a},x_{p;t_r},x_{p;t_{r+1}},x_{q;t_{r+1}},x_{q;t_{r+2}};(x_{i;t_{r+1}}^{\pm1})_{i\neq p,q}],$$
or
$$z\in\mathbb{NP}[x_{p;t_b},x_{q;t_b},x_{q;t_r},x_{p;t_{r+1}},x_{q;t_{r+1}},x_{p;t_{r+2}};(x_{i;t_{r+1}}^{\pm1})_{i\neq p,q}].$$
Then by Lemma \ref{lempositive},  the $p$-th component of ${\bf d}^{t^{\prime}}(z)$ is nonnegative and so,  ${\bf d}^{t^{\prime}}(z)$ is in $\mathbb N^n$.
\end{proof}

By Remark \ref{remklaurent} and Theorem \ref{thmdgood},  skew-symmetric cluster algebras are positive and {\bf d}-vector-positive GLP algebras. Then by Theorem \ref{mainthm} and \ref{properthm}, a new proof of the following result is indeed given, which was proved in \cite{CKLP} in the quite different method.
\begin{Corollary}(\cite{CKLP})
Any skew-symmetric cluster algebra $\mathcal A(M)$ has the  proper  Laurent property, and thus its cluster monomials are linearly independent.
\end{Corollary}

\section{ On Conjecture \ref{conjg}  }\label{conj2section}
In this section, we introduce the definition of {\bf g}-vectors for a class of GLP algebras (pointed GLP algebras), which generalizes the  {\bf g}-vectors defined for cluster algebras with principal coefficients (see \cite{FZ3}). Then we will discuss Conjecture \ref{conjg} for GLP algebras.

Denote by $\mathcal L_{p}(t)=\mathbb Z[y_1,\cdots,y_m][x_{1;t}^{\pm1},\cdots,x_{n;t}^{\pm1}],\;\;~\mathcal L_{p}^+(t)=\mathbb Z_{\geq0}[y_1,\cdots,y_m][x_{1;t}^{\pm1},\cdots,x_{n;t}^{\pm1}],$ and $$\mathbb Q_{sf}(t):=\{\frac{f}{g}| g\neq0; f,g\in \mathbb Z_{\geq0}[y_1,\cdots,y_m; x_{1;t},\cdots,x_{n;t}]\}$$
where $x_{1,t},\cdots,x_{n;t},y_1,\cdots,y_m$ are algebraically independent over $\mathbb{Z}$ for any nonnegative integer $m$.

\begin{Definition}\label{defprincipal}
Let $\mathcal A(M_T)$ be a GLP algebra.

 (i) $\mathcal A(M_T)$ is said to be of {\bf  geometric type} if $X_{t^{\prime}}\subseteq\mathcal L_{p}(t)$ for any $t,t^{\prime}\in T$.

 (ii) $\mathcal A(M_T)$ is said to be  {\bf pointed at $t_0\in T$} (or say, {\bf at cluster $X_{t_0}$}) if it is of geometric type and there exists a vertex $t_0\in T$ such that any cluster variable $x_{i;t}$ of $M_T$ has the form:
\begin{equation}\label{gvectorglp}
x_{i;t}={\bf x}_{t_0}^{{\bf g}_{i;t}^{t_0}}(1+\sum\limits_{0\neq {\bf v}\in\mathbb N^m,~~{\bf u}\in\mathbb Z^n}c_{\bf v}{\bf y}^{\bf v}{\bf x}_{t_0}^{\bf u}),\;\;\;\;\; \text{with}\;\;\;\;\; c_{\bf v}\in\mathbb Z.
\end{equation}
where  the ${\bf g}_{i;t}^{t_0}$ is called the {\bf {\bf g}-vector} of $x_{i;t}$ with respect to $X_{t_0}$ and $G_t^{t_0}=({\bf g}_{1;t}^{t_0},\cdots,{\bf g}_{n;t}^{t_0})\in M_n(\mathbb Z)$ is called the {\bf G-matrix} of $X_t$ with respect to $X_{t_0}$; the {\bf {\bf g}-vector} of a cluster monomial ${\bf x}_t^{\bf a}$ with respect to $X_{t_0}$ is defined to be $G_t^{t_0}{\bf a}$.
\end{Definition}

\begin{Remark}
The notion of {\em ``pointed"} in cluster theory has been used in \cite{LLZ,Q1,Q2}. One can refer Definition 2.2.1 in \cite{Q2} for detail.
\end{Remark}

The definition of {\bf g}-vectors for  pointed GLP algebras is inspired by the following theorem.

\begin{Theorem} (\cite{GHKK})\label{thmghkk}
Let $\mathcal A(M)$ be a skew-symmetrizable cluster algebra with principal coefficients at $t_0$. Then any cluster variable $x_{i;t}$ of  $\mathcal A(M)$ has the form of
$$x_{i;t}={\bf x}_{t_0}^{{\bf g}_{i;t}^{t_0}}(1+\sum\limits_{0\neq {\bf v}\in\mathbb N^n}c_{\bf v}{\bf y}^{\bf v}{\bf x}_{t_0}^{B_{t_0}{\bf v}}),$$
where ${\bf g}_{i;t}^{t_0}$ is the g-vector of $x_{i;t}$, ${\bf y}^{\bf v}=\prod\limits_{j=1}^ny_i^{v_i}$, and $c_{\bf v}\geq 0$.
\end{Theorem}

  A GLP algebra $\mathcal A(M_T)$ pointed at $t_0$  is called  {\bf positive} if $X_{t^{\prime}}\subseteq\mathcal L_{p}^+(t)$ for any $t,t^{\prime}\in T$ and  {\bf weakly positive} if $X_{t^{\prime}}\subseteq\mathbb Q_{sf}(t)$ for any $t,t^{\prime}\in T$. Obviously, positivity means weak positivity since $\mathcal L_{p}^+(t)\subseteq\mathbb Q_{sf}(t)$. From the definition of mutation of cluster variables, it is easy to see that cluster algebras are always weakly positive.

\begin{Theorem}\label{detgthm}
Let $\mathcal A(M_T)$ be a positive GLP algebra pointed at $t_0$, and $X_t, X_{t^{\prime}}$ be any two clusters of $\mathcal A(M_T)$, then

(i) $x_{j;t^{\prime}}$ has the following form $$x_{j;t^{\prime}}={\bf x}_{t}^{{\bf r}_{j;t}^{t^{\prime}}}(1+\sum\limits_{0\neq {\bf v}\in\mathbb N^m,~~{\bf u}\in\mathbb Z^n}c_{\bf v}{\bf y}^{\bf v}{\bf x}_t^{\bf u}),\;\;\;\; \text{with}\;\;\;\; c_{\bf v}\geq 0.$$

(ii) $G_{t^{\prime}}^{t_0}=G_t^{t_0}R_{t}^{t^{\prime}}$, where $ R_{t}^{t^{\prime}}=({\bf r}_{1;t}^{t^{\prime}},\cdots,{\bf r}_{n;t}^{t^{\prime}})$ . In particular, $det G_t^{t_0}=\pm1$.

\end{Theorem}

\begin{proof}
By the assumption, we know
 $x_{j;t^{\prime}}\in\mathcal L_{p}^+(t)$. So, $x_{j;t^{\prime}}$ has the form of
$$x_{j;t^{\prime}}=\sum\limits_{p\in I}\lambda_{p,j}{\bf x}_{t}^{{\bf r}_{p,j;t}^{t_0}}+\sum\limits_{0\neq {\bf v}\in\mathbb N^m,~~{\bf u}\in\mathbb Z^n}c_{\bf v}{\bf y}^{\bf v}{\bf x}_t^{\bf u},~~\lambda_{p,j},c_{\bf v}\geq0.$$
By viewing $x_{j;t^{\prime}}$ as an element in $\mathcal L_{p}(t_0)$, we have
\begin{eqnarray}\label{coreqa}
{\bf x}_{t_0}^{{\bf g}_{j;t^{\prime}}^{t_0}}=x_{j;t^{\prime}}|_{y_1=\cdots=y_m=0}=\sum\limits_{p\in I}\lambda_{p,j}{\bf x}_{t}^{{\bf r}_{p,j;t}^{t_0}}|_{y_1=\cdots=y_m=0}=
\sum\limits_{p\in I}\lambda_{p,j}{\bf x}_{t_0}^{G_t^{t_0}{\bf r}_{p,j;t}^{t_0}}, \lambda_{p,j}\geq 0.
\end{eqnarray}
We get that there is only one $\lambda_{p,j}$ such that $\lambda_{p,j}=1$, all the other $\lambda_{p,j}=0$, so $x_{j;t^{\prime}}$ has the form of
$$x_{j;t^{\prime}}={\bf x}_{t}^{{\bf r}_{j;t}^{t_0}}+\sum\limits_{0\neq {\bf v}\in\mathbb N^m,~~{\bf u}\in\mathbb Z^n}c_{\bf v}{\bf y}^{\bf v}{\bf x}_t^{\bf u},~~\lambda_{j},c_{\bf v}\geq 0,$$
i.e., $x_{i;t^{\prime}}$ has the following form $x_{j;t^{\prime}}={\bf x}_{t}^{{\bf r}_{j;t}^{t^{\prime}}}(1+\sum\limits_{0\neq {\bf v}\in\mathbb N^m,~~{\bf u}\in\mathbb Z^n}c_{\bf v}{\bf y}^{\bf v}{\bf x}_t^{\bf u}),~~c_{\bf v}\geq 0.$

Thus the equality (\ref{coreqa}) is just ${\bf x}_{t_0}^{{\bf g}_{j;t^{\prime}}^{t_0}}=x_{j;t^{\prime}}|_{y_1=\cdots=y_m=0}=
{\bf x}_{t_0}^{G_t^{t_0}{\bf r}_{j;t}^{t_0}}.$
So, we obtain $G_{t^{\prime}}^{t_0}=G_t^{t_0}R_{t}^{t^{\prime}}$.
Take $t^{\prime}=t_0$, we have $G_t^{t_0}R_{t}^{t_0}=G_{t_0}^{t_0}=I_n$. Because $G_t^{t_0},R_{t}^{t_0}\in M_n(\mathbb Z)$, it holds $det G_t^{t_0}=\pm1$.
\end{proof}

\begin{Remark}
By the above theorem, and the definition of  GLP algebra pointed at some cluster, we see that if a positive GLP algebra $\mathcal A(M_T)$ is pointed at $t_0$, then $\mathcal A(M_T)$ is pointed at any $t\in T$.
So, the ${\bf r}_{i;t}^{t^{\prime}}$ and $R_{t}^{t^{\prime}}$ appearing in the proof in this theorem can be re-written as ${\bf r}_{i;t}^{t^{\prime}}={\bf g}_{i;t}^{t^{\prime}}$, $R_{t}^{t^{\prime}}=G_{t}^{t^{\prime}}$.
\end{Remark}

\begin{Theorem}\label{thmmonomial}
Let $\mathcal A(M_T)$ be a  positive GLP algebra pointed  at $t_0$, and ${\bf x}_{t_1}^{\bf a}, {\bf x}_{t_2}^{\bf d}$ be  two cluster monomials of $\mathcal A(M_T)$. If ${\bf x}_{t_1}^{\bf a}$ and ${\bf x}_{t_2}^{\bf d}$ have the same {\bf g}-vector, i.e., $G_{t_1}^{t_0}{\bf a}=G_{t_2}^{t_0}{\bf d}$, then ${\bf x}_{t_1}^{\bf a}={\bf x}_{t_2}^{\bf d}$.

\end{Theorem}
\begin{proof}
Since $\mathcal A(M_T)$ is a  positive GLP algebra pointed at $t_0$, by viewing ${\bf x}_{t_1}^{\bf a}, {\bf x}_{t_2}^{\bf d}$ as elements in $\mathcal L(t_0)$, we know ${{\bf x}_{t_1}^{\bf a}|}_{y_1=\cdots=y_m=0}={\bf x}_{t_0}^{G_{t_1}^{t_0}{\bf a}}={\bf x}_{t_0}^{G_{t_2}^{t_0}{\bf d}}={{\bf x}_{t_2}^{\bf d}|}_{y_1=\cdots=y_m=0}$.
By Theorem \ref{detgthm} (i),  the Laurent expansion of the cluster monomial  ${\bf x}_{t_1}^{\bf a}$ with respect to the cluster at $t_2$  has the following form
\begin{equation}\label{equ12}
{\bf x}_{t_1}^{\bf a}={\bf x}_{t_2}^{{\bf b}}(1+\sum\limits_{0\neq {\bf v}\in\mathbb N^m,~~{\bf u}\in\mathbb Z^n}c_{\bf v}{\bf y}^{\bf v}{\bf x}_{t_2}^{\bf u}),\;\;c_{\bf v}\geq 0.
\end{equation}
 We obtain that
\begin{eqnarray}
{\bf x}_{t_0}^{G_{t_1}^{t_0}{\bf a}}={\bf x}_{t_1}^{\bf a}|_{y_1=\cdots y_m=0}&=&({\bf x}_{t_2}^{{\bf b}}(1+\sum\limits_{0\neq {\bf v}\in\mathbb N^m,~~{\bf u}\in\mathbb Z^n}c_{\bf v}{\bf y}^{\bf v}{\bf x}_{t_2}^{\bf u}))|_{y_1=\cdots=y_m=0}\nonumber\\
&=&{\bf x}_{t_0}^{G_{t_2}^{t_0}\bf b}(1+\sum\limits_{0\neq {\bf v}\in\mathbb N^m,~~{\bf u}\in\mathbb Z^n}c_{\bf v}{\bf y}^{\bf v}{\bf x}_{t_0}^{G_{t_2}^{t_0}\bf u})_{y_1=\cdots=y_m=0}\nonumber\\
&=&{\bf x}_{t_0}^{G_{t_2}^{t_0}\bf b}.\nonumber
\end{eqnarray}
Thus $G_{t_2}^{t_0}{\bf b}=G_{t_1}^{t_0}{\bf a}=G_{t_2}^{t_0}{\bf d}$, which implies ${\bf b}={\bf d}$, by Theorem \ref{detgthm} (ii).
 Then by (\ref{equ12}),  ${\bf x}_{t_1}^{\bf a}$ can be written as follows:
\begin{eqnarray}\label{adeqn}
{\bf x}_{t_1}^{\bf a}={\bf x}_{t_2}^{{\bf d}}+\sum\limits_{0\neq {\bf v_1}\in\mathbb N^m,~~{\bf u_1}\in\mathbb Z^n}c_{\bf v_1}{\bf y}^{\bf v_1}{\bf x}_{t_2}^{\bf u_1},\;\;\text{with}\;\; c_{\bf v_1}\geq 0.
\end{eqnarray}
Similarly, the Laurent expansion of the cluster monomial  ${\bf x}_{t_2}^{\bf d}$ with respect to the cluster at $t_1$  has the following form
$${\bf x}_{t_2}^{\bf d}={\bf x}_{t_1}^{{\bf a}}+\sum\limits_{0\neq {\bf v_2}\in\mathbb N^m,~~{\bf u_2}\in\mathbb Z^n}c_{\bf v_2}{\bf y}^{\bf v_2}{\bf x}_{t_1}^{\bf u_2},\;\;\text{with}\;\; c_{\bf v_2}\geq 0.$$
Thus we  obtain
$$\sum\limits_{0\neq {\bf v_1}\in\mathbb N^m,~~{\bf u_1}\in\mathbb Z^n}c_{\bf v_1}{\bf y}^{\bf v_1}{\bf x}_{t_2}^{\bf u_1}+\sum\limits_{0\neq {\bf v_2}\in\mathbb N^m,~~{\bf u_2}\in\mathbb Z^n}c_{\bf v_2}{\bf y}^{\bf v_2}{\bf x}_{t_1}^{\bf u_2}=0.$$
Clearly, $\sum\limits_{0\neq {\bf v_2}\in\mathbb N^m,~~{\bf u_2}\in\mathbb Z^n}c_{\bf v_2}{\bf y}^{\bf v_2}{\bf x}_{t_1}^{\bf u_2}$ can be written as follows:
$$\sum\limits_{0\neq {\bf v_2}\in\mathbb N^m,~~{\bf u_2}\in\mathbb Z^n}c_{\bf v_2}{\bf y}^{\bf v_2}{\bf x}_{t_1}^{\bf u_2}=\frac{f_1}{g_1},$$
 where $f_1,g_1\in\mathbb Z_{\geq 0}[y_1,\cdots,y_m,x_{1;t_1},\cdots,x_{n;t_1}]$.
Since $\mathcal A(M_T)$ is a  positive and pointed GLP algebra,  ${\bf x}_{t_2}^{\bf u_1}$ can be written as a quotient of two polynomials in $\mathbb Z_{\geq 0}[y_1,\cdots,y_m,x_{1;t_1},\cdots,x_{n;t_1}]$. Because $c_{\bf v_1}\geq 0$,  we can write  $\sum\limits_{0\neq {\bf v_1}\in\mathbb N^m,~~{\bf u_1}\in\mathbb Z^n}c_{\bf v_1}{\bf y}^{\bf v_1}{\bf x}_{t_2}^{\bf u_1}$ in the following form:
$$\sum\limits_{0\neq {\bf v_1}\in\mathbb N^m,~~{\bf u_1}\in\mathbb Z^n}c_{\bf v_1}{\bf y}^{\bf v_1}{\bf x}_{t_2}^{\bf u_1}=\frac{f_2}{g_2},$$
where $f_2,g_2\in\mathbb Z_{\geq0}[y_1,\cdots,y_m,x_{1;t_1},\cdots,x_{n;t_1}]$.
Thus
$$\sum\limits_{0\neq {\bf v_1}\in\mathbb N^m,~~{\bf u_1}\in\mathbb Z^n}c_{\bf v_1}{\bf y}^{\bf v_1}{\bf x}_{t_2}^{\bf u_1}+\sum\limits_{0\neq {\bf v_2}\in\mathbb N^m,~~{\bf u_2}\in\mathbb Z^n}c_{\bf v_2}{\bf y}^{\bf v_2}{\bf x}_{t_1}^{\bf u_2}=\frac{f_2}{g_2}+\frac{f_1}{g_1}=0,$$
then we obtain $f_1g_2+f_2g_1=0$, where $f_1,g_1,f_2,g_2\in\mathbb Z_{\geq0}[y_1,\cdots,y_m,x_{1;t_1},\cdots,x_{n;t_1}]$ and $g_1\neq0,g_2\neq0$.
So, we have $f_2=0=f_1$, i.e., $$\sum\limits_{0\neq {\bf v_1}\in\mathbb N^m,~~{\bf u_1}\in\mathbb Z^n}c_{\bf v_1}{\bf y}^{\bf v_1}{\bf x}_{t_2}^{\bf u_1}=0=\sum\limits_{0\neq {\bf v_2}\in\mathbb N^m,~~{\bf u_2}\in\mathbb Z^n}c_{\bf v_2}{\bf y}^{\bf v_2}{\bf x}_{t_1}^{\bf u_2},$$
Then, by (\ref{adeqn}), we obtain that ${\bf x}_{t_1}^{\bf a}={\bf x}_{t_2}^{\bf d}$.
\end{proof}

From  Theorem \ref{thmmonomial} and Theorem \ref{detgthm} (ii),  we obtain the following result as an affirmation of Conjecture \ref{conjg} for a {\em positive} GLP algebra pointed at $t_0$.
\begin{Corollary}\label{corconj}
Let $\mathcal A(M_T)$ be a  GLP algebra pointed at $t_0$. If $\mathcal A(M_T)$ is positive, then

(i)\; Different cluster monomials of $\mathcal A(M_T)$ have different {\bf g}-vectors.

(ii)\; The ${\bf g}$-vectors ${\bf g}_{1;t}^{t_0},\cdots, {\bf g}_{n;t}^{t_0}$ form a $\mathbb Z$-basis of $\mathbb Z^n$.
\end{Corollary}

\begin{Remark} (i) Since skew-symmetrizable cluster algebras with principal coefficients at $t_0$ are always positive GLP algebras pointed at $t_0$ by \cite{GHKK},   Conjecture \ref{conjg} holds for skew-symmetrizable cluster algebras.

 (ii) By the results in \cite{HL}, we know that the acyclic sign-skew-symmetric cluster algebras with principal coefficients are positive and pointed GLP algebras,  and then Conjecture \ref{conjg} holds for acyclic sign-skew-symmetric cluster algebras.

\end{Remark}

   Note that the affirmation of Conjecture \ref{conjg} (ii) for acyclic sign-skew-symmetric cluster algebras was also given by Min Huang and us in \cite{CHL2} via the other approach.

 As mentioned in \cite{CHL}, positivity and sign-coherence are two  deep phenomena in cluster algebras. Sign-coherence can be  equivalently understood  that the corresponding cluster algebras with principal coefficients are pointed. It was known in \cite{NZ} that sign-coherence can be related with many other properties of cluster algebras. Also, we believe that positivity can be used to explain some other properties of cluster algebras.  Through   Theorem \ref{mainthm}, \ref{detgthm}, \ref{thmmonomial} and the methods of their proofs, we attempt to provide some evidences for the essentiality of positivity and sign-coherence.

{\bf Acknowledgements:}\; This project is supported by the National Natural Science Foundation of China (No.11671350 and No.11571173).


\end{document}